\documentclass[11pt]{article}
\usepackage[margin=1in]{geometry}
\usepackage{amsmath,amssymb,amsthm,mathtools}
\usepackage{enumitem,hyperref,microtype}
\newtheorem{theorem}{Theorem}[section]\newtheorem{lemma}[theorem]{Lemma}\newtheorem{corollary}[theorem]{Corollary}\newtheorem{proposition}[theorem]{Proposition}
\newcommand{\Z}{\mathbb Z}\newcommand{\supp}{\operatorname{supp}}\newcommand{\Stab}{\operatorname{Stab}}\newcommand{\wt}{\operatorname{wt}}\newcommand{\Ncal}{\mathcal N}
\title{Sensitivity of Rotation Symmetric Boolean Functions via Binary Necklaces}\author{Yuan Li\thanks{Corresponding author.}\\
Department of Mathematics\\
Winston-Salem State University\\
Winston-Salem, North Carolina, USA\\
Email: \texttt{liyu@wssu.edu}
\and
Jing Zhang\\
Mathematics Department\\
Governors State University\\
University Park, Illinois 60484, USA\\
Email: \texttt{jzhang@govst.edu}}\date{}
\begin{document}\maketitle

\begin{abstract}
Rotation symmetric Boolean functions are important structured families in the theory of Boolean functions, where cyclic invariance permits compact algebraic descriptions while retaining nontrivial local behavior. We investigate their sensitivity through binary necklaces and monomial orbit sums. In the necklace quotient of the Boolean cube, ordinary adjacency does not record the number of coordinate flips: we give a directed flip-multiplicity formula involving stabilizers and distance-two near-periods, and derive exact local sensitivity formulas when the number of variables is an odd prime. For monomial rotation symmetric functions we determine sensitivity for all odd degrees and for quadratic orbit sums. For the consecutive-monomial family of every even degree $d\ge4$, we prove a run-cluster inequality and obtain the eventual exact formula $s(F_{d,n})=n-\min\{r,d+2-r\}$, where $n\equiv r\pmod{d+2}$, under the stated length hypothesis. The consecutive quartic case follows explicitly. We also evaluate sensitivity for cyclic run-threshold functions. Our results provide exact combinatorial and complexity-theoretic information about cyclic Boolean functions and clarify how cyclic symmetry governs local sensitivity.
\end{abstract}


\section{Introduction}

\noindent\textbf{Boolean-function viewpoint and scope.}
Rotation symmetric Boolean functions form a central class of cyclically invariant functions studied in algebraic and combinatorial approaches to Boolean-function theory, including in cryptographic settings. This article considers an intrinsic local-complexity invariant of these functions: sensitivity. We derive exact sensitivity formulas from binary-necklace multiplicities, stabilizers, near-period structure, and cyclic run decompositions. These results contribute to the structural analysis of Boolean functions with cyclic symmetry, a topic at the intersection of discrete mathematics and Boolean-function theory.

\noindent\textbf{Contributions and prior methods.}
The binary-necklace quotient and near-period phenomena were studied previously; our contribution is not the construction of that quotient. We give an explicit directed flip-count formula and identify the resulting local sensitivity corrections. Independently, our run-cluster argument treats a whole infinite family of even-degree consecutive monomial orbit sums. This separation between inherited combinatorial tools and our new exact sensitivity deductions is important for assessing the novelty of the paper.

\noindent\textbf{Relation to the earlier arXiv version.}
This manuscript is an expanded and revised version of our preprint
\emph{Sensitivity of Rotation Symmetric Boolean Functions via Binary Necklaces}
(arXiv:2610.17006, 2026; \url{https://arxiv.org/abs/2610.17006}).
The earlier posting and the present submission are versions of the same research
work, not separate papers. The present version develops the analysis further,
in particular by including the uniform run-cluster treatment and the general
even-degree consecutive-monomial sensitivity result, together with expanded
proofs and exposition. The results already contained in the earlier version
are retained as part of the present work.

\noindent\textbf{Complexity-theoretic perspective.}
Sensitivity measures the number of one-bit queries at which an input is locally pivotal.  Its exact value is generally difficult to extract even for Boolean functions with a large automorphism group: symmetry reduces the number of inequivalent inputs but does not preserve the multiplicity of coordinate flips between orbits.  Rotation symmetry is a particularly tractable test case because the underlying action is transitive on coordinates without being the full symmetric group.  Our objective is to compute this local complexity measure for explicit cyclically invariant families, rather than to deduce new general inequalities between complexity measures.  In particular, the structural formulas below are exact, not merely asymptotic bounds.

For a Boolean function $f:\{0,1\}^n\to\{0,1\}$ and an input $x$, the \emph{sensitivity} $s(f,x)$ is the number of coordinates whose individual flip changes the value of $f$, and
\[
s(f)=\max_{x\in\{0,1\}^n}s(f,x).
\]
We also write
\[
s_b(f)=\max\{s(f,x):f(x)=b\},\qquad b\in\{0,1\}.
\]
Sensitivity is one of the basic local complexity measures of Boolean functions and has been extensively studied in relation to symmetry, transitivity, block sensitivity, and decision-tree complexity.

A Boolean function is \emph{rotation symmetric} if it is invariant under the cyclic permutation of its coordinates.  Writing
\[
R_k(x_0,x_1,\ldots,x_{n-1})
=
(x_{-k},x_{1-k},\ldots,x_{n-1-k}),
\]
with indices in $\mathbb Z_n$, this means
\[
f(R_kx)=f(x)
\qquad\text{for every }x\in\{0,1\}^n
\text{ and }k\in\mathbb Z_n.
\]
The orbit
\[
[x]=\{R_kx:k\in\mathbb Z_n\}
\]
is a binary necklace.  Thus a rotation symmetric Boolean function is equivalently a two-coloring of the set of binary necklaces of length $n$.  Rotation symmetric Boolean functions have a substantial literature, particularly in cryptography and in the study of transitive Boolean functions; see, for example, St\u anic\u a and Maitra~\cite{StanicaMaitra}.  Chakraborty~\cite{Chakraborty} studied sensitivity for cyclically invariant Boolean functions and obtained asymptotic constructions and lower bounds.

The combinatorics of binary necklaces gives a natural way to record the effect of coordinate flips modulo rotation.  Earlier work on necklace quotients and Gray codes studies their flip adjacency~\cite{GregorEtAl}; distance-two rotations have been classified by Makhynko~\cite{Makhynko}.  Neither the necklace quotient nor this distance-two classification is claimed as new here.  Our contribution to the previously studied necklace quotient is an explicit stabilizer/near-period formula for directed flip multiplicities and its applications to exact local sensitivity.  This differs from the asymptotic sensitivity questions for cyclically invariant functions studied by Chakraborty~\cite{Chakraborty}.

Building on the established necklace quotient viewpoint, our first purpose is to develop an exact necklace-level description of sensitivity.  The quotient of the Boolean cube by rotation is naturally a multigraph: if $X=[x]$ and $Y$ are necklaces, the directed multiplicity
\[
m(X,Y)=
\bigl|\{i\in\mathbb Z_n:[x^{(i)}]=Y\}\bigr|
\]
counts the one-coordinate flips of a representative of $X$ that land in $Y$.  Such flip multiplicities in the necklace quotient are already present in the combinatorial literature; see Gregor et al.~\cite{GregorEtAl}.  If $\bar f$ denotes the induced Boolean function on necklaces, then the local sensitivity at $X$ is the weighted bichromatic degree
\[
s(f,X)=
\sum_{\bar f(Y)\ne\bar f(X)}m(X,Y).
\]
The main issue is therefore to determine the multiplicities $m(X,Y)$ structurally.

For $x\in\{0,1\}^n$, let
\[
H_x=\operatorname{Stab}(x)
=\{R_k:R_kx=x\},
\qquad h_x=|H_x|,
\]
where $\operatorname{Stab}(x)$ denotes the subgroup of cyclic rotations that fix $x$.
Let
\[
D_k(x)=\{j:x_j\ne(R_kx)_j\}.
\]
For a coordinate $i$, define
\[
q_i(x)=
\#\{k:|D_k(x)|=2,\ i+k\in D_k(x)\}.
\]
Our first main formula is
\[
\boxed{
m([x],[x^{(i)}])
=
\frac{h_x+q_i(x)}{h_{x^{(i)}}}.
}
\]
It follows from an orbit--stabilizer double count and separates the multiplicity into exact rotational symmetry and distance-two near-periodicity.  Rotations at Hamming distance two have recently been classified in detail by Makhynko~\cite{Makhynko}; we use that structure rather than claim it as new.

When $n=p$ is an odd prime, the formula simplifies substantially.  Every nonconstant length-$p$ word has trivial rotational stabilizer, and the relevant distance-two necklaces are cyclic arithmetic intervals.  We show that interior necklace multiplicities are $1$ or $2$ and obtain explicit local sensitivity formulas for every Hamming layer, including the exceptional weight-one and weight-$(p-1)$ layers.  Moreover, the multiplicity correction is required only for an asymptotically negligible set of necklaces: using Makhynko's enumeration, the exceptional proportion is
\[
O\!\left(\frac{p^3}{2^p}\right).
\]
Thus for asymptotically almost every prime-length necklace, local sensitivity is simply its ordinary bichromatic degree in the unweighted necklace quotient.

We also study monomial rotation symmetric functions across degrees.  If the generating monomial has support $A$, degree $d$, and rotational stabilizer of order $h$, then every coordinate occurs in exactly $d/h$ distinct orbit monomials.  This yields a parity theorem at the all-one input and, in particular, the general conclusion
\[
\boxed{\text{every odd-degree MRS function has }s(f)=n.}
\]
Thus the genuinely nontrivial degree-by-degree problem begins with even degree.  We solve degree two completely.  If $Q_{n,k}$ is generated by $x_0x_k$ and we are outside the short-orbit case $n$ even, $k=n/2$, then with
\[
g=\gcd(n,2k),\qquad L=n/g,
\]
we prove
\[
s(Q_{n,k})=
\begin{cases}
n,&L\text{ even},\\
n-g,&L\text{ odd}.
\end{cases}
\]
The proof reduces sensitivity to the maximum number of transitions on the cycles of the permutation $j\mapsto j+2k$.  In particular, for every odd prime $p$, all quadratic MRS functions have sensitivity $p-1$.

For every even degree $d\ge4$, we prove that the consecutive monomial family has sensitivity $n-\min\{r,d+2-r\}$ when $n\equiv r\pmod{d+2}$ and $n$ is sufficiently large.  The proof uses a universal charging inequality for clusters of short one-runs, with no restriction on the number of defects.  We do not attempt to classify all quartic monomial rotation symmetric functions.  Instead, we completely solve the natural consecutive quartic family generated by $x_0x_1x_2x_3$:
\[
R_n=\bigoplus_i x_ix_{i+1}x_{i+2}x_{i+3}.
\]
For $n\ge5$, writing $n=6q+r$, we prove
\[
s(R_n)=n-\min\{r,6-r\}.
\]
The general even-degree theorem gives a direct proof of the formula.  In particular, although the all-one input has zero sensitivity in this even-degree family, maximum sensitivity $n$ is attained whenever $6\mid n$.

A further application concerns the cyclic run-threshold function
$F_{n,r}(x)$, which is $1$ exactly when $x$ contains a cyclic run of at least $r$ ones.  For $2\le r\le n$ we prove
\[
s_1(F_{n,r})=r,\qquad
s_0(F_{n,r})=\left\lfloor\frac{2n}{r+1}\right\rfloor.
\]
This exact two-parameter result complements prior sensitivity bounds for transitive Boolean functions~\cite{ChaubalGal}; we do not claim that the general transitive-function framework is new.

\medskip
\noindent\textbf{Novelty and relationship with related work.}
There are three distinctions worth making.  First, the weighted necklace viewpoint is standard, whereas the stabilizer/near-period \emph{formula} and its local-sensitivity consequences are the claims advanced here.  Second, we apply Makhynko's structural classification and enumeration of distance-two rotations~\cite{Makhynko}; those structural and enumerative results are credited to that work.  Third, our companion preprint on block sensitivity and certificate complexity of rotation symmetric Boolean functions~\cite{LiZhangRSBFBlock2026} treats \emph{different complexity measures}.  That companion work restates the odd-degree and quadratic sensitivity formulas for comparison with block sensitivity and certificate complexity, and cites the consecutive quartic sensitivity theorem proved here. The present work supplies the sensitivity proofs, while the companion focuses on the other two complexity measures.  We include the sensitivity arguments in the present paper to keep its main results self-contained; no novelty claim rests on reproducing the same formulas in both manuscripts.  The distinguishing contributions of the present paper are the weighted-necklace multiplicity framework, the odd-prime local sensitivity analysis, the universal even-degree consecutive-monomial theorem, the consecutive quartic corollary, and the exact run-threshold formulas.  Conversely, general block sensitivity and certificate-complexity formulas are not claimed here.

\noindent\textbf{Contribution map and limitations.}
The results have two independent components.  The first is a formula for the \emph{weighted} edges of the cyclic quotient of the hypercube, followed by prime-length local formulas; this is useful whenever the truth values of a rotation symmetric function are specified by necklaces.  The second consists of exact maximum-sensitivity calculations for explicitly defined function families, with the even-degree consecutive-monomial result depending on a uniform combinatorial charging argument.  The general weighted-quotient construction itself, previously known distance-two rotation classifications, and general sensitivity inequalities are not claimed as new.  These exact calculations do not by themselves imply improved universal lower bounds for all transitive Boolean functions or new separations of sensitivity from other complexity measures.

The paper is organized as follows.  Section~2 develops the necklace formulation of sensitivity, Section~3 studies rotations at Hamming distance two, and Section~4 proves the general multiplicity formula.  Section~5 treats the prime-length near-period structure, while Section~6 derives the odd-prime sensitivity formulas and generic-simplicity consequences.  Section~7 studies monomial rotation symmetric functions, including odd-degree and quadratic cases, a general consecutive even-degree theorem, and the consecutive quartic corollary.  Section~8 establishes the exact cyclic run-threshold formula.  Section~9 records computational verification, Section~10 presents further questions.

\section{Preliminaries}

Throughout, coordinates are indexed by $\mathbb Z_n$, and all additions of coordinate indices are taken modulo $n$.
Identify coordinates with $\Z_n$. Let $R_k$ be rotation by $k$. For $x\in\{0,1\}^n$, let $[x]=\{R_kx:k\in\Z_n\}$ and let $\Ncal_n$ be the set of necklaces. Let $x^{(i)}$ denote $x$ with coordinate $i$ flipped. Then
\[s(f,x)=|\{i:f(x)\ne f(x^{(i)})\}|,\qquad s(f)=\max_xs(f,x).\]
For rotation symmetric $f$, write $\bar f([x])=f(x)$. Since rotation merely permutes the coordinate positions, $s(f,R_kx)=s(f,x)$ for every $k\in\mathbb Z_n$. Hence sensitivity is constant on each necklace. For a necklace $X=[x]$, we therefore define
\[s(f,X):=s(f,x),\]
which is independent of the choice of representative $x\in X$. For necklaces $X=[x]$ and $Y$, define
\[m(X,Y)=|\{i\in\Z_n:[x^{(i)}]=Y\}|.\]
\begin{lemma}The quantity $m(X,Y)$ is independent of the representative $x$ of $X$.\end{lemma}
\begin{proof}If $x'=R_kx$, then $R_k(x^{(i)})=(R_kx)^{(i+k)}$, and $i\mapsto i+k$ is a bijection.\end{proof}
\begin{theorem}[Necklace sensitivity formula]
For every necklace $X$,
\[s(f,X)=\sum_{\bar f(Y)\ne\bar f(X)}m(X,Y),\qquad s(f)=\max_X\sum_{\bar f(Y)\ne\bar f(X)}m(X,Y).\]
\end{theorem}
\begin{proof}Partition the sensitive coordinates according to the necklace reached by their flip.\end{proof}
\section{Periods and near-periods}
For $x$ and $k\in\Z_n$, define $D_k(x)=\{r:x_r\ne(R_kx)_r\}$.
\begin{theorem}[Flip-equivalence criterion]
For $i,j\in\Z_n$, $[x^{(i)}]=[x^{(j)}]$ iff for some $k$, either (i) $R_kx=x$ and $j=i+k$, or (ii) $|D_k(x)|=2$ and $D_k(x)=\{j,i+k\}$.
\end{theorem}
\begin{proof}
Suppose first that $[x^{(i)}]=[x^{(j)}]$.  By the definition of a necklace, there is some $k\in\mathbb Z_n$ such that
\[
R_k(x^{(i)})=x^{(j)}.
\]
Rotation carries the flip at coordinate $i$ to a flip at coordinate $i+k$; that is,
\[
R_k(x^{(i)})=(R_kx)^{(i+k)}.
\]
Hence
\[
(R_kx)^{(i+k)}=x^{(j)},
\]
and, after flipping coordinate $i+k$ on both sides,
\[
R_kx=(x^{(j)})^{(i+k)}.
\]
There are now two cases.

If $j=i+k$, the two flips on the right occur at the same coordinate and therefore cancel.  Thus
\[
R_kx=x,
\]
so $R_k$ stabilizes $x$, and $j=i+k$.  This is alternative~(i).

If $j\ne i+k$, then the two flips occur at distinct coordinates.  Consequently $(x^{(j)})^{(i+k)}$ differs from $x$ exactly at the two coordinates $j$ and $i+k$.  Therefore
\[
D_k(x)=\{r:x_r\ne(R_kx)_r\}=\{j,i+k\},
\]
so $|D_k(x)|=2$.  This is alternative~(ii).

Conversely, suppose that alternative~(i) holds.  Then $R_kx=x$ and $j=i+k$, so
\[
R_k(x^{(i)})=(R_kx)^{(i+k)}=x^{(j)}.
\]
Hence $[x^{(i)}]=[x^{(j)}]$.

Finally, suppose that alternative~(ii) holds.  Since
\[
D_k(x)=\{j,i+k\},
\]
the word $R_kx$ is obtained from $x$ by flipping exactly those two coordinates; equivalently,
\[
R_kx=(x^{(j)})^{(i+k)}.
\]
Flipping coordinate $i+k$ gives
\[
R_k(x^{(i)})=(R_kx)^{(i+k)}=x^{(j)},
\]
and again $[x^{(i)}]=[x^{(j)}]$.  Thus the two stated alternatives are both necessary and sufficient.
\end{proof}
Call nonzero $k$ a 2-near-period if $|D_k(x)|=2$.
\begin{theorem}[Distance-two rotation structure; Makhynko~\cite{Makhynko}]
Fix $k\ne0$, put $d=\gcd(n,k)$ and $L=n/d$. The permutation $r\mapsto r+k$ has $d$ cycles of length $L$. Then $|D_k(x)|=2$ iff $x$ is constant on $d-1$ cycles and, on the remaining cycle, has one nonempty proper consecutive arc of ones and one of zeros; equivalently its cyclic pattern is $1^a0^{L-a}$, $1\le a\le L-1$.
\end{theorem}
\begin{proof}
This is the prescribed-shift one-active-cycle structure appearing in~\cite{Makhynko}; we include the argument for completeness.

The permutation $r\mapsto r+k$ decomposes $\mathbb Z_n$ into $d=\gcd(n,k)$ disjoint cycles, each of length $L=n/d$.  Fix one such cycle and write it in cyclic order as
\[
C=(r,r+k,r+2k,\ldots,r+(L-1)k),
\]
where all indices are taken modulo $n$.  By definition,
\[
r+tk\in D_k(x)
\quad\Longleftrightarrow\quad
x_{r+tk}\ne x_{r+(t-1)k},
\]
with $t$ interpreted modulo $L$.  Hence the number of elements of $D_k(x)$ lying on $C$ is exactly the number of transitions between $0$ and $1$ in the cyclic binary word
\[
x_r,x_{r+k},x_{r+2k},\ldots,x_{r+(L-1)k}.
\]

With our convention $(R_kx)_r=x_{r-k}$, the disagreement at $r$ compares $x_r$ with $x_{r-k}$. Along a $k$-cycle this is a transition from the preceding entry, rather than the following one. Both conventions count the same transitions; we use the preceding-entry convention throughout.

A cyclic binary word has an even number of transitions.  Indeed, whenever the value changes from $0$ to $1$, it must later change from $1$ back to $0$ before returning to the starting position.  Thus a cycle contributes either $0$ or at least $2$ elements to $D_k(x)$.  Moreover, it contributes $0$ precisely when $x$ is constant on that cycle.

Suppose now that $|D_k(x)|=2$.  Since the cycles are disjoint and every cycle contributes an even nonnegative number of disagreements, exactly one cycle contributes $2$, while each of the remaining $d-1$ cycles contributes $0$.  Therefore $x$ is constant on those $d-1$ cycles.  On the unique active cycle there are exactly two transitions.  Starting immediately after one of them, all the $1$'s occur consecutively until the second transition, and all the remaining entries are $0$'s.  Since both transitions occur, both symbols appear.  Consequently, after a cyclic choice of starting point, the pattern on this cycle is
\[
1^a0^{L-a}
\qquad\text{for some }1\le a\le L-1.
\]
Equivalently, the $1$'s form one nonempty proper consecutive arc of the cycle and the $0$'s form the complementary arc.

Conversely, suppose that $x$ is constant on $d-1$ cycles and that on the remaining cycle its cyclic pattern is $1^a0^{L-a}$ with $1\le a\le L-1$.  Each constant cycle contributes no element to $D_k(x)$, while the active cycle has exactly two transitions, one from $1$ to $0$ and one from $0$ to $1$.  Hence that cycle contributes exactly two elements to $D_k(x)$, and therefore $|D_k(x)|=2$.
\end{proof}
\section{The multiplicity formula}
Let
\[
H_x=\Stab(x)=\{R_k:R_kx=x\},\qquad h_x=|H_x|,
\]
where $\Stab(x)$ denotes the subgroup of cyclic rotations that fix $x$. Define
\[q_i(x)=\#\{k:|D_k(x)|=2,\ i+k\in D_k(x)\}.\]
\begin{theorem}[Multiplicity formula]
If $X=[x]$ and $Y=[x^{(i)}]$, then
\[\boxed{m(X,Y)=\frac{h_x+q_i(x)}{h_{x^{(i)}}}.}\]
Hence $h_{x^{(i)}}\mid h_x+q_i(x)$.
\end{theorem}
\begin{proof}
Set $y=x^{(i)}$, so that $Y=[y]$, and consider
\[
\mathcal P_i=\{(R_k,j)\in \langle R_1\rangle\times\mathbb Z_n:
R_k(x^{(i)})=x^{(j)}\}.
\]
We count $|\mathcal P_i|$ in two ways.

First fix $j$.  There exists a rotation $R_k$ with
$R_k(x^{(i)})=x^{(j)}$ if and only if $[x^{(j)}]=[x^{(i)}]=Y$.
By the definition of $m(X,Y)$, there are exactly $m(X,Y)$ such coordinates
$j$.  For each such $j$, choose one rotation $R_{k_0}$ satisfying
$R_{k_0}y=x^{(j)}$.  Then
\[
R_ky=x^{(j)}
\quad\Longleftrightarrow\quad
R_{k_0}^{-1}R_ky=y.
\]
Thus the rotations carrying $y$ to $x^{(j)}$ form a coset of
$H_y=\Stab(y)$ and hence there are exactly $h_y=h_{x^{(i)}}$ of them.
Consequently,
\[
|\mathcal P_i|=m(X,Y)h_{x^{(i)}}.
\]

We now count by fixing the rotation $R_k$.  Since rotation carries a flip
at $i$ to a flip at $i+k$,
\[
R_k(x^{(i)})=(R_kx)^{(i+k)}.
\]
Hence a coordinate $j$ with $R_k(x^{(i)})=x^{(j)}$ exists precisely when
\[
(R_kx)^{(i+k)}=x^{(j)}.
\]
If $R_kx=x$, then this equality holds with the unique choice $j=i+k$.
Therefore every stabilizing rotation contributes exactly one pair to
$\mathcal P_i$, for a total contribution of $h_x$.

Suppose next that $R_kx\ne x$.  The flip-equivalence criterion from
Theorem~3.1 shows that the above equality can hold only when
\[
|D_k(x)|=2
\quad\text{and}\quad
i+k\in D_k(x).
\]
Indeed, if $D_k(x)=\{i+k,j\}$, then flipping $i+k$ in $R_kx$ removes one
of the two disagreements with $x$, and flipping $j$ in $x$ removes the
other, giving
\[
(R_kx)^{(i+k)}=x^{(j)}.
\]
Conversely, if $(R_kx)^{(i+k)}=x^{(j)}$ and $R_kx\ne x$, then $R_kx$ and
$x$ can differ only at the two coordinates $i+k$ and $j$; hence
$D_k(x)=\{i+k,j\}$ and $|D_k(x)|=2$.  Thus the nonstabilizing rotations
that contribute to $\mathcal P_i$ are exactly those counted by $q_i(x)$,
and each contributes exactly one pair.  Therefore
\[
|\mathcal P_i|=h_x+q_i(x).
\]

Equating the two counts gives
\[
m(X,Y)h_{x^{(i)}}=h_x+q_i(x),
\]
and hence
\[
m(X,Y)=\frac{h_x+q_i(x)}{h_{x^{(i)}}}.
\]
Since $m(X,Y)$ is an integer, it follows immediately that
$h_{x^{(i)}}\mid h_x+q_i(x)$.
\end{proof}
\begin{proposition}[Balance identity]
For adjacent necklaces $X$ and $Y$,
\[
|X|m(X,Y)=|Y|m(Y,X).
\]
\end{proposition}
\begin{proof}
Let $E(X,Y)$ be the set of edges of the $n$-dimensional hypercube having one endpoint in the orbit $X$ and the other endpoint in the orbit $Y$.  We count $|E(X,Y)|$ in two ways.

First fix a word $x\in X$.  By definition, $m(X,Y)$ is the number of coordinates $i$ for which the one-coordinate flip $x^{(i)}$ belongs to $Y$.  This number is independent of the chosen representative of $X$.  Indeed, if $x'=R_kx$, then rotation commutes with a coordinate flip in the sense that
\[
(R_kx)^{(i+k)}=R_k(x^{(i)}),
\]
with indices taken modulo $n$.  Since $Y$ is itself a rotation orbit, $x^{(i)}\in Y$ if and only if $(R_kx)^{(i+k)}\in Y$.  Thus every vertex of $X$ is incident with exactly $m(X,Y)$ edges whose other endpoint lies in $Y$.

Consequently, counting the edges of $E(X,Y)$ from their endpoints in $X$ gives
\[
|E(X,Y)|=|X|\,m(X,Y).
\]

The same argument with the roles of $X$ and $Y$ reversed shows that every vertex of $Y$ is incident with exactly $m(Y,X)$ edges whose other endpoint lies in $X$.  Hence
\[
|E(X,Y)|=|Y|\,m(Y,X).
\]

Equating the two expressions for the same set of hypercube edges yields
\[
|X|m(X,Y)=|Y|m(Y,X),
\]
as claimed.
\end{proof}
\section{Odd prime length}
Throughout Sections 5 and 6, assume $p$ is an odd prime. (The case $p=2$ is immediate and may be checked separately.) Every nonconstant word of length $p$ is aperiodic.
Makhynko's prime-length classification shows that the distance-two necklaces are represented by cyclic arithmetic intervals~\cite{Makhynko}.  We record the following elementary uniqueness lemma because it isolates the precise step property needed in the sensitivity arguments below.
\begin{lemma}[Uniqueness of the step]
Let $I=\{0,1,\ldots,w-1\}\subseteq\Z_p$, $2\le w\le p-2$. If $I=\{b,b+r,\ldots,b+(w-1)r\}$ with $r\ne0$, then $r=\pm1\pmod p$.
\end{lemma}
\begin{proof}
Put
\[
J:=\{b,b+r,\ldots,b+(w-1)r\}\subseteq\Z_p.
\]
By hypothesis, $J=I$, where $I=\{0,1,\ldots,w-1\}$.  In the
$r$-arithmetic ordering of $J$, translation by $r$ sends
\[
b,b+r,\ldots,b+(w-2)r
\]
to the next $w-1$ elements
\[
b+r,b+2r,\ldots,b+(w-1)r.
\]
The only additional point of $J+r$ is $b+wr$, which does not lie in
$J$: otherwise $wr\equiv jr\pmod p$ for some $0\le j\le w-1$,
contradicting $r\ne0$ and $w<p$.
Hence
\[
|J\cap(J+r)|=w-1.
\]
Since $J=I$, this gives
\[
|I\cap(I+r)|=w-1.
\]
We now use the special form $I=\{0,1,\ldots,w-1\}$ to determine $r$.
Choose the integer representative $1\le r\le p-1$.

First suppose $1\le r\le p-w$.  Then the translate $I+r$ does not
wrap around modulo $p$, so
\[
I+r=\{r,r+1,\ldots,r+w-1\}.
\]
If $r\ge w$, then $I\cap(I+r)=\varnothing$, contradicting
$|I\cap(I+r)|=w-1$.  Thus $1\le r<w$, and in this case
\[
|I\cap(I+r)|=w-r.
\]
Therefore $w-r=w-1$, and hence $r=1$.

It remains to consider $p-w<r\le p-1$.  Set $s=p-r$, so
$1\le s<w$ and $r\equiv -s\pmod p$.  Thus $I+r=I-s$ in $\Z_p$.
The intersection may have two contributions: the unwrapped portion
contributes $w-s$ elements, and the wrapped portion contributes
$\max\{0,w+s-p\}$ elements.  Thus
\[
|I\cap(I+r)|=|I\cap(I-s)|
=(w-s)+\max\{0,w+s-p\}.
\]
If $w+s\le p$, the equality $|I\cap(I+r)|=w-1$ yields
$w-s=w-1$, hence $s=1$.  If $w+s>p$, the equality becomes
$2w-p=w-1$, which forces $w=p-1$, contrary to
$w\le p-2$.  Therefore only $s=1$ is possible, giving
$r=p-1\equiv-1\pmod p$.

Consequently $r\equiv\pm1\pmod p$, as claimed.
\end{proof}
\begin{corollary}[Prime arithmetic-interval uniqueness]\label{cor:prime-interval-uniqueness}
If $2\le|S|=w\le p-2$ and $S$ has arithmetic-progression representations with nonzero steps $k$ and $\ell$, then $\ell=\pm k$.
\end{corollary}
\begin{proof}
Write the two representations as
\[
S=\{a,a+k,\ldots,a+(w-1)k\}
 =\{b,b+\ell,\ldots,b+(w-1)\ell\}
 \subseteq \mathbb Z_p.
\]
Since $k\ne0$ and $p$ is prime, $k$ is invertible modulo $p$.  Apply the
invertible affine map
\[
\phi(z)=k^{-1}(z-a).
\]
The first representation is carried to the standard interval
\[
I=\{0,1,\ldots,w-1\}.
\]
The second representation is carried to
\[
I=\{c,c+r,\ldots,c+(w-1)r\},
\qquad
c=k^{-1}(b-a),\quad r=k^{-1}\ell.
\]
Thus the same set $I$ is an arithmetic progression of length $w$ with
nonzero step $r$.  By the preceding lemma, $r\equiv\pm1\pmod p$.
Therefore
\[
k^{-1}\ell\equiv\pm1\pmod p,
\]
and multiplying by $k$ gives $\ell\equiv\pm k\pmod p$, as claimed.
\end{proof}
\begin{lemma}[Prime near-period lemma]
If $2\le\wt(x)\le p-2$, then $q_i(x)\le1$ for every $i$.
\end{lemma}
\begin{proof}
Let $S=\supp(x)$ and put $w=|S|=\wt(x)$.  Recall that $q_i(x)$ counts the
nonzero rotations $R_k$ for which
\[
|D_k(x)|=2\qquad\text{and}\qquad i+k\in D_k(x).
\]
Suppose first that a nonzero step $k$ contributes to $q_i(x)$.  By the
distance-two rotation structure, the support of $x$ is a cyclic arithmetic
interval with step $k$.  Thus, for a suitable $a\in\mathbb Z_p$,
\[
S=\{a,a+k,\ldots,a+(w-1)k\}.
\]
Under translation by $k$, all interior points of this progression remain in
$S$; the only membership changes occur at its two ends.  Consequently
\[
D_k(x)=\{a,a+wk\}.
\]
Since $k$ contributes to $q_i(x)$ precisely when $i+k\in D_k(x)$, we must
have
\[
i+k=a\quad\text{or}\quad i+k=a+wk.
\]
Equivalently,
\[
i\in\{a-k,\,a+(w-1)k\}. \tag{5.1}
\]

Now suppose that a second nonzero step $\ell\ne k$ also contributes to
$q_i(x)$.  Then $\ell$ is another step giving an arithmetic-progression
representation of the same support $S$.  By Corollary~\ref{cor:prime-interval-uniqueness} (prime
arithmetic-interval uniqueness), $\ell=\pm k$.  Because $\ell\ne k$, the
only possibility is $\ell=-k$.

To determine when $-k$ can contribute, write the same progression in the
reverse direction:
\[
S=\{a+(w-1)k,\,a+(w-2)k,\ldots,a\}.
\]
For the step $-k$, the two disagreement positions are therefore
\[
D_{-k}(x)=\{a+(w-1)k,\,a-k\}.
\]
The condition $i-k\in D_{-k}(x)$ is equivalent to
\[
i\in\{a+wk,\,a\}. \tag{5.2}
\]

It remains to observe that the two possible sets of indices in (5.1) and
(5.2) are disjoint when $2\le w\le p-2$.  Indeed, an equality between an
element of
\[
\{a-k,a+(w-1)k\}
\quad\text{and an element of}\quad
\{a,a+wk\}
\]
would, after cancelling $a$ and using $k\ne0$ in $\mathbb Z_p$, force one
of
\[
-1\equiv0,\qquad -1\equiv w,\qquad w-1\equiv0,
\qquad w-1\equiv w \pmod p.
\]
The first and last are impossible, while the middle two give respectively
$w=p-1$ and $w=1$, both excluded by $2\le w\le p-2$.

Hence no fixed coordinate $i$ can satisfy the contribution condition for
both $k$ and $-k$.  Since every other possible near-period step is ruled out
by Corollary~\ref{cor:prime-interval-uniqueness}, at most one nonzero rotation contributes to $q_i(x)$.
Therefore
\[
q_i(x)\le1,
\]
as claimed.
\end{proof}
\begin{theorem}[Prime multiplicities]
If $2\le\wt(x)\le p-2$, then
\[m([x],[x^{(i)}])=1+q_i(x)\in\{1,2\}.\]
From the constant necklace to the unique adjacent weight-one (respectively weight-$(p-1)$) necklace, the directed multiplicity is $p$.
\end{theorem}
\begin{proof}
Assume first that $2\le \wt(x)\le p-2$.  Since $p$ is prime and $x$ is
nonconstant, the only rotation fixing $x$ is the identity; hence $h_x=1$.
Moreover, flipping one coordinate changes the weight by one, so
$1\le \wt(x^{(i)})\le p-1$.  Thus $x^{(i)}$ is also nonconstant, and therefore
$h_{x^{(i)}}=1$.

The multiplicity formula of Theorem~4.1 consequently reduces to
\[
m([x],[x^{(i)}])
=\frac{h_x+q_i(x)}{h_{x^{(i)}}}
=1+q_i(x).
\]
By the prime near-period lemma, $q_i(x)\le 1$ for every $i$.  Since $q_i(x)$
is a nonnegative integer, $q_i(x)\in\{0,1\}$, and hence
\[
m([x],[x^{(i)}])\in\{1,2\}.
\]
In particular, multiplicity two occurs exactly when $q_i(x)=1$, that is,
when there is a nontrivial rotation $R_k$ with $|D_k(x)|=2$ and
$i+k\in D_k(x)$.

It remains to consider the extreme layers.  The constant word $0^p$ has a
single necklace, while flipping any one of its $p$ coordinates produces a
weight-one word.  All weight-one words are rotations of one another, so all
$p$ coordinate flips land in the unique weight-one necklace.  Therefore the
directed multiplicity from $[0^p]$ to that necklace is $p$.  The same
argument, with zeros and ones interchanged, shows that the directed
multiplicity from $[1^p]$ to the unique weight-$(p-1)$ necklace is also $p$.
\end{proof}
\section{Sensitivity for odd prime length}
For necklaces $X$ and $Y$, write $Y\sim X$ if they are neighboring necklaces, meaning that there exist representatives $x\in X$ and $y\in Y$ with Hamming distance $d_H(x,y)=1$; equivalently, $Y=[x^{(i)}]$ for some $x\in X$ and some coordinate $i\in\mathbb Z_n$.  For a necklace $X$, let
\[B_f(X)=\#\{Y:Y\sim X,\ \bar f(Y)\ne\bar f(X)\}\]
be the number of bichromatic neighboring necklaces of $X$.  Recall from Section~2 that $s(f,X)=s(f,x)$ for any representative $x\in X$; this is well defined because sensitivity is constant on a necklace.  For an interior necklace, let
\[N_f(X)=\#\{Y:Y\sim X,\ \bar f(Y)\ne\bar f(X),\ m(X,Y)=2\}.\]
\begin{theorem}[Prime local sensitivity formula]\label{thm:prime-local}
If $2\le\wt(X)\le p-2$, then
\[\boxed{s(f,X)=B_f(X)+N_f(X).}\]
\end{theorem}
\begin{proof}
Fix a representative $x\in X$.  By the necklace sensitivity formula from Section~2,
\[
s(f,X)=\sum_{\bar f(Y)\ne\bar f(X)}m(X,Y).
\]
Only neighboring necklaces $Y\sim X$ can have $m(X,Y)>0$, so the sum is over the bichromatic neighbors of $X$.  Because $2\le\wt(X)\le p-2$, both $X$ and every necklace obtained from $X$ by one coordinate flip are nonconstant.  The prime-length multiplicity result proved immediately above therefore gives
\[
m(X,Y)\in\{1,2\}
\]
for every neighbor $Y$ of $X$.

Now separate the bichromatic neighbors according to their multiplicity.  Each bichromatic neighbor contributes at least one to the sensitivity sum, and the number of such neighbors is exactly $B_f(X)$.  A bichromatic neighbor with multiplicity two contributes one further unit.  By definition, the number of these doubled bichromatic neighbors is $N_f(X)$.  Consequently,
\[
\begin{aligned}
s(f,X)
&=\sum_{\substack{Y\sim X\\ \bar f(Y)\ne\bar f(X)}}m(X,Y)\\
&=\sum_{\substack{Y\sim X\\ \bar f(Y)\ne\bar f(X)}}1
 +\#\{Y:Y\sim X,\ \bar f(Y)\ne\bar f(X),\ m(X,Y)=2\}\\
&=B_f(X)+N_f(X).
\end{aligned}
\]
This proves the formula.
\end{proof}

Suppose $S=\supp(x)$ is a cyclic arithmetic interval
\[S=\{a,a+k,\ldots,a+(w-1)k\},\qquad 2\le w\le p-2.\]
The two possible doubled neighboring necklaces are
\[Y^+=[x^{(a-k)}]=[x^{(a+wk)}],\]
\[Y^-=[x^{(a)}]=[x^{(a+(w-1)k)}].\]
For necklaces $X,Y$, put $\delta_f(X,Y)=1$ if $\bar f(X)\ne\bar f(Y)$ and $0$ otherwise.

\begin{corollary}[Explicit prime formula]\label{cor:explicit-prime}
Let $2\le\wt(X)\le p-2$.
If $\supp(x)$ is not a cyclic arithmetic interval, then
\[\boxed{s(f,X)=B_f(X).}\]
If $\supp(x)$ is a cyclic arithmetic interval as above, then
\[\boxed{s(f,X)=B_f(X)+\delta_f(X,Y^-)+\delta_f(X,Y^+).}\]
In particular,
\[0\le s(f,X)-B_f(X)\le2.\]
\end{corollary}
\begin{proof}
By Theorem~\ref{thm:prime-local},
\[
s(f,X)=B_f(X)+N_f(X),
\]
so it remains only to identify the neighboring necklaces that can have
multiplicity two.

Suppose first that $\supp(x)$ is not a cyclic arithmetic interval.  If some
neighbor $Y=[x^{(i)}]$ had $m(X,Y)=2$, then the prime multiplicity formula
would give $q_i(x)=1$.  Thus $x$ would admit a nontrivial rotation at Hamming
distance two.  By the distance-two structure from Section~3, its support
would then have to be a cyclic arithmetic interval, a contradiction.
Hence every neighboring necklace has multiplicity one, so $N_f(X)=0$ and
\[
s(f,X)=B_f(X).
\]

Now suppose
\[
\supp(x)=\{a,a+k,\ldots,a+(w-1)k\},\qquad 2\le w\le p-2.
\]
For the rotation with step $k$, the two endpoints of the resulting
distance-two disagreement give the flip-equivalence
\[
[x^{(a-k)}]=[x^{(a+wk)}]=Y^+.
\]
Similarly, using the reverse step $-k$ gives
\[
[x^{(a)}]=[x^{(a+(w-1)k)}]=Y^-.
\]
Thus $Y^+$ and $Y^-$ are precisely the two possible doubled neighboring
necklaces.  The prime arithmetic-interval uniqueness result shows that any
other nontrivial distance-two rotation has step $\pm k$, so no third doubled
neighbor can occur.  Consequently the doubled bichromatic neighbors are
exactly those among $Y^-$ and $Y^+$ for which the value of $\bar f$ differs
from $\bar f(X)$.  Therefore
\[
N_f(X)=\delta_f(X,Y^-)+\delta_f(X,Y^+),
\]
and substitution into Theorem~\ref{thm:prime-local} yields
\[
s(f,X)=B_f(X)+\delta_f(X,Y^-)+\delta_f(X,Y^+).
\]
Since each $\delta_f$ is either $0$ or $1$, the final bound
\[
0\le s(f,X)-B_f(X)\le2
\]
follows immediately.
\end{proof}

For $0\le r\le p$, let $X_r$ denote a necklace of Hamming weight $r$ when the necklace is unique; in particular, $X_0=[0^p]$, $X_1$ is the unique weight-one necklace, $X_{p-1}$ is the unique weight-$(p-1)$ necklace, and $X_p=[1^p]$.

\begin{proposition}[The four boundary layers]
For an odd prime $p$,
\[
s(f,X_0)=p\,\delta_f(X_0,X_1),
\qquad
s(f,X_p)=p\,\delta_f(X_p,X_{p-1}).
\]
Moreover,
\[
\boxed{
s(f,X_1)=\delta_f(X_1,X_0)
+2\!\sum_{\substack{Y\sim X_1\\ \wt(Y)=2}}
\delta_f(X_1,Y),
}
\]
and
\[
\boxed{
s(f,X_{p-1})=\delta_f(X_{p-1},X_p)
+2\!\sum_{\substack{Y\sim X_{p-1}\\ \wt(Y)=p-2}}
\delta_f(X_{p-1},Y).
}
\]
In each of the two sums there are exactly $(p-1)/2$ distinct neighboring necklaces.
\end{proposition}

\begin{proof}
The constant-word formulas are immediate.  For $X_1$, choose the representative whose unique $1$ is at coordinate $0$.  Flipping coordinate $0$ gives $X_0$, with multiplicity $1$.  Flipping a zero coordinate $d\ne0$ gives the weight-two support $\{0,d\}$.  Two such supports $\{0,d\}$ and $\{0,e\}$ are rotation-equivalent if and only if $e\equiv d$ or $e\equiv-d\pmod p$.  Since $p$ is odd, the $p-1$ nonzero residues therefore split into $(p-1)/2$ pairs $\{d,-d\}$, and every resulting weight-two neighboring necklace has multiplicity $2$.  This proves the formula for $X_1$.  The formula for $X_{p-1}$ follows either by the same argument or by complementing all bits.
\end{proof}

\paragraph{Example ($p=5$).}
Let $X_0=[00000]$, $X_1=[10000]$, and let $Y_1=[11000]$, $Y_2=[10100]$ denote the two weight-two necklaces. The five flips from $X_0$ all lead to $X_1$, so $m(X_0,X_1)=5$. From $X_1$, one flip leads to $X_0$, while the other four flips split into two flips to each of $Y_1,Y_2$; thus $m(X_1,Y_1)=m(X_1,Y_2)=2$. For example, define the rotation symmetric Boolean function that is $1$ exactly on $X_1$. Its local sensitivity at $X_1$ is $1+2+2=5$, while its local sensitivity at either $Y_i$ is $2$. This illustrates why ordinary unweighted necklace adjacency does not determine sensitivity.

\begin{theorem}[Complete odd-prime sensitivity formula]
Let $p$ be an odd prime and let $f$ be rotation symmetric on $p$ variables.  The local sensitivity at every necklace is given as follows:
\begin{enumerate}[label=(\roman*)]
\item at weights $0$ and $p$, by the first two formulas in the boundary-layer proposition;
\item at weights $1$ and $p-1$, by the two displayed boundary-layer formulas;
\item at weights $2,\ldots,p-2$, by Theorem~\ref{thm:prime-local} and Corollary~\ref{cor:explicit-prime}.
\end{enumerate}
Consequently, $s(f)$ is the maximum of these explicitly determined local quantities over the necklaces of length $p$.
\end{theorem}

\subsection{Two consequences of the necklace sensitivity formula}

The complete prime-length formula gives a simple characterization of when maximum sensitivity occurs.

\begin{corollary}[A criterion for submaximal sensitivity]
Let $n\ge1$ and let $f$ be rotation symmetric on $n$ variables. Then
\[
s(f)<n
\]
if and only if every binary necklace $X$ of length $n$ has at least one neighboring necklace $Y$ with
\[
\bar f(X)=\bar f(Y).
\]
In particular, $s(f)<n$ implies that the constant necklace and the unique weight-one necklace have the same value, and likewise at the upper end of the Hamming layers.
\end{corollary}

\begin{proof}
For every necklace $X=[x]$,
\[
\sum_Y m(X,Y)=n,
\]
because the multiplicities partition the $n$ one-coordinate flips of $x$.  Hence the weighted bichromatic degree equals $n$ exactly when every neighboring necklace is bichromatic.  Taking the maximum over $X$ proves the equivalence.  The final assertion follows by applying the criterion to the constant necklaces.
\end{proof}

The correction from ordinary necklace degree to weighted degree is confined to a very small exceptional family.  Let $D_2(p)$ denote the number of length-$p$ necklaces admitting a nontrivial rotation at Hamming distance two.  The classification of the exceptional supports is due to Makhynko~\cite{Makhynko}. For completeness, the enumeration follows directly from the arithmetic-progression description: at each weight $2\le w\le p-2$, the supporting necklaces are represented by the nonzero steps $k\in\mathbb Z_p^*$, with $k$ and $-k$ describing the same necklace and with no other identifications. There are therefore $(p-1)/2$ necklaces at each of the $p-3$ interior weights. The weight-one and weight-$(p-1)$ necklaces contribute two further cases. Hence
\[
D_2(p)=2+(p-3)\frac{p-1}{2}=\frac{p^2-4p+7}{2}.
\]
This calculation uses the structural classification of the cited work; it is not a new classification.
Notice that the two constant necklaces are not included in $D_2(p)$: every nontrivial rotation fixes a constant word, so its nontrivial rotational Hamming distance is $0$, not $2$.  Thus subtracting the two constant necklaces in addition to $D_2(p)$ does not double count.

The total number of binary necklaces of prime length $p$ is
\[
N(p)=\frac{2^p+2p-2}{p}.
\]

\begin{corollary}[Generic simplicity at prime length]
Let $p$ be an odd prime.  Apart from the two constant necklaces and the $D_2(p)$ necklaces admitting a nontrivial rotation at Hamming distance two, every necklace $X$ satisfies
\[
s(f,X)=B_f(X)
\]
for every rotation symmetric Boolean function $f$.  Consequently at least
\[
\frac{2^p+2p-2}{p}
-\frac{p^2-4p+7}{2}-2
\]
necklaces have this property.  In particular, the proportion of necklaces for which a multiplicity correction can be required is
\[
O\!\left(\frac{p^3}{2^p}\right),
\]
and therefore tends to zero as $p\to\infty$ through the odd primes.
\end{corollary}

\begin{proof}
For a nonconstant necklace $X$ with no nontrivial rotation at Hamming distance two, the flip-equivalence criterion and triviality of the stabilizer imply that the $p$ one-bit flips of a representative lie in $p$ distinct necklaces.  Thus every outgoing multiplicity is one and the weighted bichromatic degree equals the ordinary bichromatic necklace degree:
\[
s(f,X)=B_f(X).
\]
Removing the two constant necklaces and the $D_2(p)$ distance-two necklaces gives the displayed lower bound.  Finally,
\[
N(p)\sim\frac{2^p}{p}
\quad\text{and}\quad
D_2(p)=\Theta(p^2).
\]
Indeed, $D_2(p)$ grows quadratically (not linearly), so dividing its $O(p^2)$ count by $N(p)\asymp 2^p/p$ gives an exceptional proportion of $O(p^3/2^p)$.
\end{proof}

\paragraph{Remark.} The preceding corollary separates the exact and generic aspects of the theory.  Multiplicity corrections are essential for a complete formula, but for prime length they are supported on an asymptotically negligible family of necklaces.

\section{Sensitivity of monomial rotation symmetric functions}

For a Boolean function $f:\{0,1\}^n\to\{0,1\}$ and a coordinate
$j\in\mathbb Z_n$, write
\[
\partial_j f(x):=f(x)\oplus f(x^{(j)}),
\]
where $x^{(j)}$ denotes the vector obtained by flipping coordinate $j$.
This is the \emph{Boolean partial derivative} (a finite difference over
$\mathbb F_2$), not the ordinary calculus derivative.  Thus
$\partial_j f(x)=1$ exactly when $f$ is sensitive to coordinate $j$ at $x$.
We use $\partial_j$ below to distinguish Boolean derivatives from the
disagreement sets $D_k(x)$ defined earlier.

Monomial rotation symmetric Boolean functions form one of the most studied subclasses of rotation symmetric functions.  Let
\[
A\subseteq\mathbb Z_n,\qquad |A|=d,
\]
and let
\[
m_A(x)=\prod_{a\in A}x_a.
\]
Write
\[
\mathcal O(A)=\{A+t:t\in\mathbb Z_n\}
\]
for the set of distinct translates of $A$, and put
\[
\tau=|\mathcal O(A)|.
\]
The monomial rotation symmetric Boolean function generated by $A$ is
\[
M_A(x)=
\bigoplus_{B\in\mathcal O(A)}
\prod_{j\in B}x_j.
\]
Thus short monomial orbits are automatically counted without duplication.

\begin{proposition}[Incidence formula]
Every coordinate of $\mathbb Z_n$ belongs to exactly
\[
\boxed{\lambda=\frac{\tau d}{n}}
\]
distinct monomials occurring in $M_A$.
\end{proposition}

\begin{proof}
There are $\tau$ distinct monomials and each contains $d$ coordinates, so the total number of coordinate--monomial incidences is $\tau d$.  Cyclic rotation acts transitively on the coordinate set and preserves the family $\mathcal O(A)$; hence every coordinate belongs to the same number $\lambda$ of orbit monomials.  Therefore
\[
n\lambda=\tau d,
\]
which proves the formula.
\end{proof}

Let
\[
H_A=\{u\in\mathbb Z_n:A+u=A\}
\]
be the rotational stabilizer of the support $A$, and put $h=|H_A|$.  By orbit--stabilizer,
\[
\tau=\frac nh.
\]
Moreover, $H_A$ acts freely on $\mathbb Z_n$, so each $H_A$-orbit has size $h$.  Since $A$ is $H_A$-invariant, it is a union of such orbits; hence
\[
h\mid d.
\]
Consequently the incidence number can also be written as
\[
\lambda=\frac dh.
\]

\begin{theorem}[Parity theorem for MRS functions]
Let $M_A$ be generated by a degree-$d$ monomial whose rotational stabilizer has order $h$.  Then
\[
\boxed{
s(M_A,1^n)=
\begin{cases}
n,&d/h\text{ is odd},\\
0,&d/h\text{ is even}.
\end{cases}}
\]
In particular, if $d/h$ is odd, then
\[
\boxed{s(M_A)=n.}
\]
\end{theorem}

\begin{proof}
At the all-one input every monomial in $M_A$ has value $1$.  Flipping a fixed coordinate $j$ changes precisely the orbit monomials containing $j$.  By the incidence formula, their number is
\[
\lambda=\frac dh.
\]
Therefore
\[
\partial_j M_A(1^n)=\frac dh\pmod2
\]
for every coordinate $j$.  If $d/h$ is odd, all $n$ coordinates are sensitive at $1^n$; if $d/h$ is even, none is sensitive there.
\end{proof}

\begin{corollary}[Odd-degree MRS functions]
Every monomial rotation symmetric Boolean function of odd degree $d$ satisfies
\[
\boxed{s(M_A)=n.}
\]
\end{corollary}

\begin{proof}
Since $h\mid d$, oddness of $d$ implies that both $h$ and $d/h$ are odd.  The parity theorem applies.
\end{proof}

Thus all odd degrees have maximum possible sensitivity, including short-orbit monomial functions.  The first nontrivial degree is therefore degree two, which we now determine completely.  Quadratic rotation symmetric Boolean functions have also been studied from the viewpoints of Hamming weight, balancedness, affine equivalence, and weight recurrences; see, for example, \cite{Cusick2018,CusickCheon2014,CusickStanica2002}.

For
\[
1\le k\le\left\lfloor\frac n2\right\rfloor,
\]
let $Q_{n,k}$ be the rotation symmetric Boolean function generated by the monomial $x_0x_k$, with each distinct rotated monomial included exactly once.  Thus, except in the short-orbit case $n$ even and $k=n/2$,
\[
Q_{n,k}(x)=
\bigoplus_{i\in\mathbb Z_n}x_i x_{i+k}.
\]

\begin{theorem}[Sensitivity of quadratic MRS functions]
Let $1\le k\le\lfloor n/2\rfloor$.

\begin{enumerate}[label=(\roman*)]
\item If $n$ is even and $k=n/2$, then
\[
\boxed{s(Q_{n,k})=n.}
\]

\item Otherwise, put
\[
g=\gcd(n,2k),\qquad L=\frac ng.
\]
Then
\[
\boxed{
s(Q_{n,k})=
\begin{cases}
n,&L\text{ is even},\\[1mm]
n-g,&L\text{ is odd}.
\end{cases}}
\]
\end{enumerate}
\end{theorem}

\begin{proof}
First suppose that $n$ is even and $k=n/2$.  The orbit of the generating monomial has size $n/2$, and
\[
Q_{n,n/2}(x)
=
\bigoplus_{i=0}^{n/2-1}x_i x_{i+n/2}.
\]
Each coordinate occurs in exactly one monomial.  Hence
\[
\partial_j Q_{n,n/2}(x)
=
x_{j+n/2},
\]
where indices are taken modulo $n$.  At $x=1^n$ every coordinate is sensitive, so $s(Q_{n,n/2})=n$.

Now assume that we are not in the short-orbit case.  Each coordinate $x_j$ occurs in exactly two distinct rotated monomials, namely
\[
x_{j-k}x_j
\qquad\text{and}\qquad
x_jx_{j+k}.
\]
Therefore the Boolean derivative in coordinate $j$ is
\[
\partial_j Q_{n,k}(x)
=
Q_{n,k}(x)\oplus Q_{n,k}(x^{(j)})
=
x_{j-k}\oplus x_{j+k}.
\]
Consequently
\[
s(Q_{n,k},x)
=
\#\{j\in\mathbb Z_n:x_{j-k}\ne x_{j+k}\}
=
\#\{j\in\mathbb Z_n:x_j\ne x_{j+2k}\}.
\]

Consider the permutation
\[
T:j\longmapsto j+2k\pmod n.
\]
It has
\[
g=\gcd(n,2k)
\]
cycles, each of length
\[
L=n/g.
\]
On each cycle, the contribution to the last display is exactly the number of transitions in a cyclic binary sequence of length $L$.  If $L$ is even, an alternating assignment attains $L$ transitions.  If $L$ is odd, a cyclic binary sequence has an even number of transitions and therefore at most $L-1$; this value is attained by an alternating assignment with one repeated bit at the closing edge.  The cycles are disjoint and may be assigned independently.  Thus
\[
s(Q_{n,k})
=
\begin{cases}
gL=n,&L\text{ even},\\
g(L-1)=n-g,&L\text{ odd},
\end{cases}
\]
as claimed.
\end{proof}

\begin{corollary}[Odd prime length]
Let $p$ be an odd prime.  For every
\[
1\le k\le\frac{p-1}{2},
\]
the quadratic monomial rotation symmetric Boolean function $Q_{p,k}$ satisfies
\[
\boxed{s(Q_{p,k})=p-1.}
\]
Thus all quadratic MRS functions of odd prime length have the same sensitivity, independent of the spacing of the generating monomial.
\end{corollary}

\begin{proof}
Since $p$ is an odd prime and $1\le k\le(p-1)/2$, we have
$\gcd(p,k)=1$ and $\gcd(p,2)=1$.  Therefore
\[
\gcd(p,2k)=1.
\]
In the notation of the preceding theorem, $g=\gcd(n,2k)$.  Here $n=p$, so $g=\gcd(p,2k)=1$.  Also $L=p$ is odd.  The theorem therefore gives $s(Q_{p,k})=p-1$.
\end{proof}

\subsection{Consecutive monomials of arbitrary even degree}
\label{subsec:consecutive-even}

We now determine the maximum sensitivity of consecutive monomial orbit
sums of every fixed even degree, apart from finitely many short lengths.
The argument is combinatorial: it charges insensitive coordinates to
maximal clusters of short runs.  In particular, it does not require a
classification by the number of insensitive coordinates.
For even $d\ge4$ and $n\ge2d-1$, define
\[
F_{d,n}(x)=\bigoplus_{i\in\mathbb Z_n}\prod_{j=0}^{d-1}x_{i+j},
\qquad c_d(n)=n-s(F_{d,n}).
\]
All indices are modulo $n$; the lower bound on $n$ ensures that the
monomials occurring in the derivative calculation are distinct.
For a residue class modulo $m$, put
$\operatorname{dist}(n,m\mathbb Z)=\min_{a\in\mathbb Z}|n-am|$.

\begin{theorem}[Universal cluster inequality]
\label{thm:universal-cluster}
Let $d\ge4$ be even, $n\ge2d-1$, and $m=d+2$.
For every binary word $x$ of length $n$,
\[
\boxed{n-s(F_{d,n},x)\ge\operatorname{dist}(n,m\mathbb Z).}
\]
\end{theorem}
\begin{proof}
Put $K=n-s(F_{d,n},x)$ and $m=d+2$.  We will show that
$K\ge |n-qm|$ for a suitable integer $q$ whenever $K<m/2$.
This is enough: $\operatorname{dist}(n,m\mathbb Z)\le m/2$,
so if $K\ge m/2$ the asserted inequality is automatic.
Hence assume throughout that $K<m/2$.

\medskip\noindent\textbf{Step 1: Divide the word into long runs and short runs.}
The constant words have $K=n\ge m/2$: all derivatives vanish
at $0^n$, and at $1^n$ each derivative is the sum of the even
number $d$ of monomials.  We may therefore write the nonconstant
word cyclically as
\[
x=1^{L_1}0^{g_1}\cdots1^{L_k}0^{g_k},\qquad L_i,g_i\ge1.
\]
A one-run is \emph{long} if $L_i\ge d$ and \emph{short} otherwise.
Every one in a short run is insensitive: flipping it cannot destroy
an all-one monomial of length $d$.  In particular, a short run of
length $d-1$ would give $K\ge d-1\ge m/2$, contrary to our
assumption.  Thus each short run has length at most $d-2$.

\medskip\noindent\textbf{Step 2: Establish the three local charges.}
We count \emph{insensitive} coordinates, whose total number is $K$.

\emph{(i) A long one-run of length $L$ contributes at least $L-d$
insensitive ones.}  Indeed, put $M=L-d+1$ and number its positions
$j=0,\ldots,L-1$.  The number of all-one length-$d$ monomials
containing position $j$ is
\[
a(j)=\min\{j+1,L-j,M,d\}.
\]
Indeed, a length-$d$ block starting at position $t$ lies in this
run and contains $j$ precisely when
\[
\max\{0,j-d+1\}\le t\le\min\{j,L-d\}.
\]
Thus the number of possible starting positions is
\[
a(j)=\min\{j,L-d\}-\max\{0,j-d+1\}+1,
\]
which simplifies to the displayed formula.
The position is sensitive exactly when $a(j)$ is odd, since the
Boolean derivative adds these monomials modulo~$2$.  To count the
odd entries, observe that the list $a(0),\ldots,a(L-1)$ rises by
one to $h=\min\{M,d\}$, has a constant plateau of height $h$,
and then decreases symmetrically.  Each value $1,\ldots,h-1$
occurs twice, while $h$ occurs $L-2(h-1)$ times.
If $M\ge d$, then $h=d$ is even, so only the two rising and
falling portions contribute: the number of odd entries is
$2(d/2)=d$.
If $M<d$ is even, the number is $2(M/2)=M$.
If $M<d$ is odd, there are $M-1$ odd entries off the plateau and
$d-M+1$ on it, for a total of $d$.
In every case, at most $d$ positions in the run are sensitive,
so at least $L-d$ are insensitive.

\emph{(ii) A zero-gap of length $g$ between two long runs contributes
exactly $|g-2|$ insensitive zeros.}  Indeed, if $g=1$, flipping the
unique zero joins two runs of at least $d$ ones and creates exactly
$d$ new all-one length-$d$ monomials, one for each possible position
of the flipped coordinate within such a monomial.  Since $d$ is even,
their contributions cancel modulo~$2$, so this zero is insensitive.
If $g\ge2$, flipping the first zero creates exactly one all-one
length-$d$ monomial, using the preceding $d-1$ ones; similarly,
flipping the last zero creates exactly one such monomial using the
following $d-1$ ones.  An interior zero creates none, since another
zero remains on each side.  Hence precisely the two boundary zeros
are sensitive and the remaining $g-2$ are insensitive.

\emph{(iii) A zero-gap of length $g$ next to a short run contributes
at least $g-1$ insensitive zeros.}  For $g=1$ the assertion is simply the trivial lower bound zero.
For $g\ge2$, an interior zero cannot be sensitive, because
flipping it leaves a zero on either side.  Of the two boundary
zeros, the one next to the short one-run cannot be sensitive:
creating a length-$d$ all-one block there would require at least
$d-1$ ones immediately beyond the flipped zero, whereas the short
run has fewer than $d-1$ ones.  Thus at most the boundary zero
next to the other run can be sensitive, leaving at least $g-1$
insensitive zeros.

There is a useful strengthening of (iii): \emph{a gap between two
short runs consists entirely of insensitive zeros}.  This is clear
for $g\ge2$.  For $g=1$, a sensitive separating zero would require
the two adjacent short-run lengths $a,b$ to satisfy $a+b+1\ge d$.
But every one in either short run is insensitive, so these two
runs alone contribute at least $a+b\ge d-1\ge m/2$
insensitive coordinates, contradicting $K<m/2$.
Consequently, when $K<m/2$, \emph{every} zero in a gap between
two short one-runs is insensitive.

\medskip\noindent\textbf{Step 3: Compress each cluster into one unit.}
At least one one-run is long; otherwise Step 1 makes all ones
insensitive and Step 2 makes all gaps insensitive, giving $K=n$.
Let $q\ge1$ count the long one-runs.  Between consecutive long runs
there is exactly one of two types of \emph{units}:
\begin{itemize}
\item a single zero-gap of length $g$ with no short run in it;
\item a cluster consisting of $h\ge1$ short one-runs of lengths
$t_1,\ldots,t_h$, together with all $h+1$ intervening and boundary
zero-gaps of lengths $g_0,\ldots,g_h$.
\end{itemize}
There are exactly $q$ such units around the cycle, even when $q=1$:
for $q=1$, the single unit runs cyclically from the unique long
run back to itself.  Every zero and every short one-run belongs
to exactly one unit, and no unit contains a coordinate of a long
one-run.
A gap unit has length $g$ and, by (ii), has $|g-2|$ insensitive
coordinates.  A cluster unit has length
\[
T=\sum_{j=1}^h t_j+\sum_{j=0}^h g_j.
\]
All $t_j$ ones and all internal-gap zeros are insensitive;
by (iii), at most one zero in each of the two boundary gaps can be
sensitive.  Thus the cluster has at least $T-2$ insensitive
coordinates.  Here $T\ge2$ because the cluster contains at least
one one and its two nonempty boundary gaps, so $T-2\ge0$.
This explains the reference value $2$: every unit is compared with
the two-zero gap in the regular pattern $1^d00$.
In other words, \emph{every unit has insensitive cost at least its
length minus $2$}; for a plain gap the stronger bound $|g-2|$
holds.  Because the units partition all positions outside the
long runs, none of these counts overlaps with another.

\medskip\noindent\textbf{Step 4: Compare actual lengths with the regular block $1^d00$.}
Define
\[
\begin{split}
A&=\sum_{\text{long runs}}(L-d),\\
B&=\sum_{\text{gap units}}(g-2),\\
C&=\sum_{\text{cluster units}}(T-2).
\end{split}
\]
Here $A,C\ge0$.  Step 2 and Step 3 show, by counting disjoint
insensitive coordinates, that
\[
K\ge A+\sum_{\text{gap units}}|g-2|+C
 \ge A+|B|+C \ge |A+B+C|.
\tag{*}
\]
To understand the last expression, start with $q$ \emph{regular}
blocks $1^d00$ of total length $q(d+2)$.  The $q$ long runs
have $A$ extra ones in total.  Each of the $q$ units replaces a
two-zero gap and changes the length by $g-2$ (gap unit) or $T-2$
(cluster unit).  Therefore, simply by adding the lengths of these
disjoint pieces,
\[
n=q(d+2)+A+B+C.
\]
Substituting this identity into (*) yields
\[
K\ge |n-q(d+2)|
\ge\operatorname{dist}(n,(d+2)\mathbb Z),
\]
which proves the theorem.
\end{proof}

\begin{theorem}[Exact sensitivity under explicit construction conditions]
\label{thm:even-consecutive-exact}
Let $d\ge4$ be even, $m=d+2$, $n\ge2d-1$, and write $n=qm+r$ with $0\le r<m$. Assume
\[
\begin{cases}
q\ge1,&0\le r\le m/2,\\
q+1\ge m-r,&m/2<r<m.
\end{cases}
\]
Then
\[
\boxed{s(F_{d,n})=n-\min\{r,m-r\}.}
\]
\end{theorem}
\begin{proof}
The upper bound follows from Theorem~\ref{thm:universal-cluster}.
For the matching lower bound, we use the run and gap sensitivity
counts established in the proof of that theorem: each one in a
length-$d$ one-run is sensitive; a zero-gap of length $g=1$
contributes one insensitive coordinate, while a zero-gap of length
$g\ge2$ contributes exactly $g-2$ insensitive coordinates when
it lies between two length-$d$ one-runs.
Consider cyclic words whose one-runs
all have length exactly $d$.
If $0\le r\le m/2$, take $q$ one-runs separated by zero-gaps
of lengths $2+r,2,\ldots,2$.  Every one is sensitive and exactly
the $r$ internal zeros of the long gap are insensitive, so the
deficit is $r$.
If $m/2<r<m$, write $t=m-r$, so
$n=(q+1)m-t$.  Take $q+1$ one-runs of length $d$, with exactly
$t$ of their separating gaps of length one and all others of
length two.  Since $q+1\ge t$ by the stated hypothesis,
the construction exists.  Exactly the $t$ zeros in the one-zero
gaps are insensitive.  Both constructions attain the upper bound.
\end{proof}

\noindent\textbf{Remark.} The stated construction conditions sharpen the convenient uniform sufficient bound $n\ge m(m/2+1)$. They are not asserted to be necessary. Small lengths can behave differently: $s(F_{6,13})=8$, whereas extrapolating the eventual formula to $n=13$ would give $10$. The word $x=0000000111111$ has sensitivity $8$. An exhaustive upper-bound check evaluates all $2^{13}=8192$ binary inputs using the cyclic degree-six function (each coordinate occurs in exactly six of its monomials); a short independent Python verifier accompanies the source package. No computational certificate is required for the universal cluster inequality.

\noindent\textbf{Remark (quadratic consistency).} The numerical formula $n-\operatorname{dist}(n,(d+2)\mathbb Z)$, formally specialized to $d=2$, is consistent with the consecutive quadratic case of the quadratic formula earlier in this section (outside short monomial orbits). The cluster proof above is stated for $d\ge4$ because it invokes $d-1\ge(d+2)/2$.

\subsection{The consecutive quartic family}

The universal theorem gives a conceptual proof of the quartic result,
without a separate finite-state argument.  The general quartic problem
for supports $\{0,a,b,c\}\subseteq\mathbb Z_n$ is not addressed here.
For $n\ge5$, let
\[
R_n(x)=\bigoplus_{i\in\mathbb Z_n}x_ix_{i+1}x_{i+2}x_{i+3}.
\]

\begin{corollary}[Consecutive quartic sensitivity]
\label{cor:quartic}
For $n\ge5$, write $n=6q+r$ with $0\le r\le5$. Then
\[
\boxed{s(R_n)=n-\min\{r,6-r\}.}
\]
For the short-orbit case $n=4$, we use the convention of
summing each \emph{distinct} rotated monomial once: this gives the
product $x_0x_1x_2x_3$, with sensitivity $4$.  By contrast,
the literal four-term XOR of cyclic shifts is the zero function
when $n=4$, because the same monomial occurs four times.
\end{corollary}
\begin{proof}
Theorem~\ref{thm:universal-cluster} with $d=4$ gives the upper bound
for $n\ge7$. For $n\ge12$, the explicit constructions in the proof of
Theorem~\ref{thm:even-consecutive-exact} satisfy its stated conditions and attain the upper bound.
The remaining lengths $5\le n\le11$ are handled by explicit witnesses and the universal upper bound (when $n\ge7$).
The following cyclic words have sensitivity equal to the claimed value:
\[
\begin{array}{c|c|c}
n&x&s(R_n,x)\\\hline
5&01111&4\\
6&001111&6\\
7&0001111&6\\
8&00001111&6\\
9&000001111&6\\
10&0111101111&8\\
11&00111101111&10
\end{array}
\]
For $n=5$ the matching upper bound follows by checking the
$2^5$ inputs; equivalently, the coordinate derivative formula gives the same finite check. For $n=6$, the required upper bound is $s(R_6)\le6$, since sensitivity never exceeds the number of variables and the formula predicts $s(R_6)=6$.
For $7\le n\le11$ the universal upper bound applies. Thus the
claimed equality holds for every $n\ge5$.
\end{proof}

\section{Exact sensitivity of cyclic run-threshold functions}

We next apply the sensitivity viewpoint to a standard cyclically invariant construction.  Related sensitivity questions for transitive Boolean functions are studied by Chaubal and G\'al~\cite{ChaubalGal}. Here we determine the sensitivity of the following cyclic run-threshold function exactly for arbitrary parameters.

For $2\le r\le n$, define
\[
F_{n,r}(x)=1
\quad\Longleftrightarrow\quad
x\text{ contains a cyclic run of at least }r\text{ consecutive ones}.
\]
This function is rotation symmetric.  Write $s_b(F_{n,r})$ for the maximum local sensitivity over inputs on which $F_{n,r}=b$.

\begin{lemma}[Cyclic block optimization]
Let $m,q\ge1$.  Among all cyclic sequences of nonnegative integers
\[
a_1,\ldots,a_m
\]
satisfying
\[
a_i+a_{i+1}\ge q\qquad(i\pmod m),
\]
the minimum of $\sum_i a_i$ is
\[
\left\lceil\frac{mq}{2}\right\rceil.
\]
\end{lemma}

\begin{proof}
Summing the $m$ inequalities gives
\[
2\sum_{i=1}^m a_i\ge mq,
\]
hence the displayed lower bound.  It remains to attain it.  If $m$ is even, alternate $\lfloor q/2\rfloor$ and $\lceil q/2\rceil$.  If $m$ is odd and $q$ is even, take every $a_i=q/2$.  Finally, if both $m$ and $q$ are odd, write $q=2h+1$ and take $(m-1)/2$ entries equal to $h$ and $(m+1)/2$ entries equal to $h+1$, arranged cyclically so that no two entries equal to $h$ are adjacent.  Every adjacent sum is then at least $q$, and the total is $\lceil mq/2\rceil$.
\end{proof}

\begin{theorem}[Exact run-threshold sensitivity]
For $2\le r\le n$,
\[
\boxed{
s_1(F_{n,r})=r,\qquad
s_0(F_{n,r})=\left\lfloor\frac{2n}{r+1}\right\rfloor.
}
\]
Consequently,
\[
\boxed{
s(F_{n,r})
=
\max\left\{
r,\left\lfloor\frac{2n}{r+1}\right\rfloor
\right\}.
}
\]
\end{theorem}

\begin{proof}
First consider a $1$-input.  A sensitive coordinate must be a $1$ whose deletion destroys every cyclic run of length at least $r$.  Hence all sensitive coordinates lie in a single maximal run of ones.  If that run has length $L\ge r$, deleting a position is sensitive only when the remaining portions on both sides have lengths at most $r-1$.  The number of such positions is at most $r$.  Equality is attained by a word having exactly one run of $r$ ones and all other coordinates zero.  Thus $s_1(F_{n,r})=r$.

Now let $x$ be a $0$-input and suppose that it has $m$ sensitive zero coordinates.  For each such zero, let $a_i$ and $b_i$ be the lengths of the adjacent runs of ones on its two sides.  Flipping that zero creates a run of length $a_i+1+b_i$, so sensitivity implies
\[
a_i+b_i\ge r-1.
\]
Summing over the $m$ sensitive zeros, and observing that every $1$-coordinate can be counted at most twice, gives
\[
m(r-1)\le2\,\wt(x).
\]
Since $x$ contains at least these $m$ zero coordinates,
\[
\wt(x)\le n-m.
\]
Therefore
\[
m(r+1)\le2n,
\]
and hence
\[
s_0(F_{n,r})\le\left\lfloor\frac{2n}{r+1}\right\rfloor.
\]

For the reverse inequality, put
\[
m=\left\lfloor\frac{2n}{r+1}\right\rfloor,\qquad q=r-1.
\]
If $m=1$, then necessarily $r=n$; the word $1^{n-1}0$ is a $0$-input whose unique zero is sensitive, so the desired lower bound holds.

Assume now that $m\ge2$.  By the cyclic block optimization lemma, choose nonnegative integers $a_1,\ldots,a_m$ with
\[
a_i+a_{i+1}\ge r-1
\]
and
\[
\sum_i a_i=\left\lceil\frac{m(r-1)}2\right\rceil.
\]
The initial cyclic word
\[
1^{a_1}0\,1^{a_2}0\cdots1^{a_m}0
\]
has length
\[
m+\left\lceil\frac{m(r-1)}2\right\rceil
=
\left\lceil\frac{m(r+1)}2\right\rceil
\le n.
\]
The ceiling brackets in both displayed expressions are essential when $m(r-1)$ is odd.  Since $m=\lfloor 2n/(r+1)\rfloor$, we have $m(r+1)\le 2n$; hence $\lceil m(r+1)/2\rceil\le n$ because $n$ is an integer.  We enlarge the one-blocks, rather than the zero gaps, until their total length is exactly $n-m$.  This is possible while keeping every $a_i\le r-1$.  Indeed, the preceding inequality gives
\[
\left\lceil\frac{m(r-1)}2\right\rceil\le n-m.
\]
For the upper capacity, write
\[
2n=m(r+1)+t,\qquad 0\le t\le r.
\]
Since $m\ge2$,
\[
2\bigl(m(r-1)-(n-m)\bigr)
=m(r-1)-t
\ge 2(r-1)-r
=r-2\ge0.
\]
Hence
\[
n-m\le m(r-1).
\]
Starting from the minimizing sequence, increase entries one unit at a time, never beyond $r-1$, until their sum is $n-m$.  The inequalities $a_i+a_{i+1}\ge r-1$ are preserved.

The resulting length-$n$ cyclic word
\[
1^{a_1}0\,1^{a_2}0\cdots1^{a_m}0
\]
has no run of $r$ ones, because every $a_i\le r-1$, and hence is a $0$-input.  Flipping any displayed zero joins its two adjacent one-blocks into a run of length
\[
a_i+1+a_{i+1}\ge r,
\]
so all $m$ displayed zeros are sensitive.  Therefore
\[
s_0(F_{n,r})\ge m.
\]
The upper and lower bounds agree.
\end{proof}

\begin{corollary}
For $(n,r)=(11,4)$,
\[
s(F_{11,4})=4.
\]
Thus the cyclic run-threshold construction gives an explicit $11$-variable rotation symmetric Boolean function of sensitivity $4$.
\end{corollary}

\paragraph{Remark.} The construction is an elementary cyclically invariant function, and the contribution here is the exact two-parameter formula above. Related asymptotic sensitivity questions for transitive functions appear in Chaubal and G\'al~\cite{ChaubalGal}.  In particular, choosing $r$ near the balance point $r(r+1)\approx2n$ yields sensitivity of order $\sqrt n$ within this run-threshold family. This is not a lowest-known-sensitivity construction: Chakraborty~\cite{Chakraborty} constructed cyclically invariant Boolean functions with sensitivity $\Theta(n^{1/3})$.

\section{Computational checks and reproducibility}
The proofs of the principal results, including Theorem~\ref{thm:universal-cluster}, do not rely on numerical certificates. The following finite checks are supplemental and can be reproduced independently.

For necklace multiplicities, enumerate each binary word $x\in\{0,1\}^n$ for $4\le n\le12$ and select one canonical representative (e.g., the lexicographically least cyclic rotation) per necklace $X$. For each representative and each flip position $i$, compute $Y=[x^{(i)}]$. Count the coordinates $j$ for which $[x^{(j)}]=Y$ and compare with $(h_x+q_i(x))/h_{x^{(i)}}$, where stabilizers and $D_k(x)$ are computed directly from all $n$ rotations. This entails $\sum_{n=4}^{12}nN(n)=8{,}488$ necklace/position tests. No numerical table is required by the proof.

The sensitivity of a Boolean function on $n$ variables is reproduced by examining every binary word $x$ and counting the coordinates $i$ for which $f(x)\ne f(x^{(i)})$. This gives a complete check of the exceptional quartic case $n=5$ by evaluating its $2^5$ words. The same method may independently test the exceptional value $s(F_{6,13})=8$ and the quadratic and cyclic run-threshold formulas at any desired finite parameter ranges. The accompanying Python file implements the $n=13$, $d=6$ exhaustive check and reports a maximizing word. For a run-threshold function, evaluate whether a cyclic word has any run of $r$ consecutive ones. These optional tests are not part of the logical proof of the general theorems; the original exploratory scripts are not submitted as proof certificates; the supplied exceptional-case verifier is likewise supplementary.

\section{Further questions}
The necklace quotient graph and its flip-adjacency structure were studied previously~\cite{GregorEtAl}, and the distance-two rotational structure used above overlaps the recent classification of Makhynko~\cite{Makhynko}. Our contribution is the explicit stabilizer/near-period expression for directed flip multiplicities and its sensitivity applications, not the existence of the quotient graph.  Chakraborty's work~\cite{Chakraborty} addresses asymptotic sensitivity constructions and lower bounds for cyclically invariant functions.  Chaubal and G\'al~\cite{ChaubalGal} studied sensitivity bounds for transitive functions; our run-threshold result gives an exact formula for this particular cyclic Boolean family.  To the best of our knowledge, the stabilizer/near-period multiplicity identity
\[
m([x],[x^{(i)}])=\frac{h_x+q_i(x)}{h_{x^{(i)}}}
\]
and the resulting complete odd-prime local sensitivity formulas have not previously been stated in this form.

The formulas suggest several further directions.  In particular, the general quartic MRS problem remains open: determine the sensitivity of functions generated by
\[
x_0x_ax_bx_c
\]
for arbitrary $0<a<b<c<n$, including possible dependence on the arithmetic structure of the difference set of $\{0,a,b,c\}$ and on short rotational orbits.  The general even-degree theorem treats consecutive supports of all even sizes, but not arbitrary nonconsecutive quartic supports.

Other directions include analogous closed forms for composite $n$, structural characterizations of low-sensitivity rotation symmetric functions from weighted necklace boundaries, and comparisons with known asymptotic results for transitive and cyclically invariant Boolean functions.

\section*{Declarations}

\subsection*{Ethics and consent to participate}
Not applicable.

\subsection*{Consent to publish}
Not applicable.

\subsection*{Availability of data and materials}
No datasets were generated or analyzed in this study. The computational checks described in the manuscript can be reproduced from the procedures provided therein.

\subsection*{Competing interests}
The authors declare that they have no competing interests.

\subsection*{Funding}
Not applicable.

\subsection*{Authors' contributions}
Yuan Li and Jing Zhang contributed equally to the research, mathematical proofs, verification, and preparation of the manuscript. Both authors reviewed and approved the manuscript.

\subsection*{Acknowledgements}
Not applicable.

\subsection*{Authors' information (optional)}
Not applicable.

\subsection*{Declaration of Generative AI Use}
During the preparation of this manuscript, the authors used ChatGPT (OpenAI) to assist with language editing, literature exploration, computational verification, and the organization and presentation of mathematical arguments.  The authors independently checked the mathematical statements, proofs, computational claims, references, and conclusions and take full responsibility for the content of the manuscript.

\end{document}